\documentclass[12pt,centertags,oneside]{amsart}

\usepackage{amscd,amsxtra,calc}
\usepackage{cmmib57}
\usepackage{url}

\usepackage{amscd}
\usepackage{pstricks}
\usepackage{color}
\usepackage[a4paper,width=16.2cm,top=3cm,bottom=3cm]{geometry}

\numberwithin{equation}{section}

\theoremstyle{plain}
\newtheorem{thm}{Theorem}[section]
\newtheorem{theorem}[thm]{Theorem}

\newtheorem{lemma}[thm]{Lemma}

\newtheorem{proposition}[thm]{Proposition}
\theoremstyle{definition}
\newtheorem{remark}[thm]{Remark}

\newtheorem{definition}[thm]{Definition}
\newtheorem{claim}[thm]{Claim}

\newtheorem{question}[thm]{Question}
\newtheorem{construction}[thm]{Construction}

\numberwithin{equation}{section}

\newcommand{\Spec}{{\rm Spec \,}}

\newcommand{\C}{{\mathbb C}}

\newcommand{\F}{{\mathbb F}}

\renewcommand{\P}{{\mathbb P}}

\newcommand{\R}{{\mathbb R}}

\newcommand{\Z}{{\mathbb Z}}

\newcommand{\id}{{\rm id\hspace{.1ex}}}

\newcommand{\Aut}{{\rm Aut\hspace{.1ex}}}

\title [non-finitely generated discrete automorphism group]{Smooth rational projective varieties with non-finitely generated discrete automorphism group and infinitely many real forms}

\author{Tien-Cuong Dinh}
\address{Department of Mathematics, National University of Singapore, 10, 
Lower Kent Ridge Road, Singapore 119076}
\email{matdtc@nus.edu.sg}
\author{Keiji Oguiso}
\address{Mathematical Sciences, the University of Tokyo, Meguro Komaba 3-8-1, Tokyo, Japan, and National Center for Theoretical Sciences, Mathematics Division, National Taiwan University, 
Taipei, Taiwan}
\email{oguiso@ms.u-tokyo.ac.jp}
\author{Xun Yu}
\address{Center for Applied Mathematics, Tianjin University, 92 Weijin Road, Nankai District,
Tianjin 300072, P. R. China.
}
\email{xunyu@tju.edu.cn}
\thanks{The first named author is supported by the Tier 2 grant MOE-T2EP20120-
0010. The second named author is supported by JSPS Grant-in-Aid (S) 15H05738, JSPS Grant-in-Aid 20H00111, 20H01809, and by NCTS Scholar Program. The third named author is supported by NSFC (No.12071337, No. 11701413, and No. 11831013).}
\subjclass[2010]{14J50, 14P05}

\begin{document}

\maketitle

\begin{abstract}
We show, among other things, that for each integer $n \ge 3$, there is a smooth complex projective rational variety of dimension $n$, with discrete non-finitely generated automorphism group and with infinitely many mutually non-isomorphic real forms. Our result is inspired by the work of Lesieutre and the work of Dinh and Oguiso. 
\end{abstract}

\section{Introduction}

It is quite recent that negative answers are given to the following long standing natural questions (see eg. \cite{BS64}, \cite{DIK00}, \cite{Kh02}, \cite{CF20} for positive directions):
\begin{question}\label{quest1} Let $V$ be a smooth complex projective variety of dimension $\ge 2$. 
\begin{enumerate}
\item Is the automorphism group $\Aut (V)$ finitely generated if $\Aut (V)$ is discrete? 
\item Are real forms of $V$, i.e., systems of homogeneous equations with real coefficients defining $V$,  finite up to isomorphisms over $\R$?
\end{enumerate}
\end{question}
The first negative answers to these questions are given by Lesieutre 
\cite{Le18}. He constructs a smooth complex  projective variety $V$ of dimension $6$ with Kodaira dimension $\kappa (V) = -\infty$ denying both (1) and (2). This variety is not rationally connected. 
Expanding his idea, Dinh and Oguiso (\cite{DO19}) construct a smooth complex projective variety $V$ of any dimension $\ge 2$ with $\kappa (V) \ge 0$ again denying both (1) and (2). In somewhat different directions, Dubouloz, Freudenburg, Moser-Jauslin construct smooth {\it affine} rational varieties for any dimension $\ge 4$ with infinitely many real forms (\cite{DFMJ21}). However, it is still completely open if there are counterexamples among smooth complex {\it projective rational} varieties, the most basic varieties in birational algebraic geometry.

The aim of this paper is to construct a smooth complex projective rational variety $V$ of any dimension $\ge 3$ denying both (1) and (2) (Theorem \ref{thm1} below). 

Before stating our main results, we recall precise definitions of crucial notions relevant to Question \ref{quest1} and our main results. 

\begin{definition}\label{def1}
\begin{enumerate} 
\item A variety of dimension $n$ is called rational if it is birational to the projective space $\P^n$ over the base field. 
\item An $\R$-scheme $W \to \Spec \R$ is called a real form of a $\C$-scheme $V \to \Spec \C$ if $W \times_{\Spec \R} \Spec \C \to \Spec \C$ is isomorphic to $V \to \Spec \C$ over $\Spec \C$. Two real forms $W_i \to \Spec \R$ ($i=1$, $2$) are isomorphic if they are isomorphic over $\Spec \R$. By abuse of language, we sometimes say that a $\C$-scheme $V$ is defined over $\R$ when a real form $W$ of $V$ is understood from the context. (See \cite{Se02} and \cite[Sect.2-4]{CF20} for more details about real forms.)
\item Let $V \to \Spec \C$ be a complex projective variety. Then the automorphism group 
$$\Aut (V) := \Aut (V/\Spec \C)$$ 
of $V$ over $\Spec \C$ has a natural locally algebraic group structure with at most countably many connected components, via the Hilbert scheme of $V \times V$. We denote by $\Aut_0(V)$ the identity component of $\Aut (V)$. It is of dimension $\dim H^0(V, T_V)$ when $V$ is smooth. Here, $T_V$ denotes the tangent bundle of $V$. So it is natural to ask if the group $\Aut (V)/\Aut_0(V)$ is always finitely generated or not. We say that $\Aut(V)$ is discrete if $\Aut_0 (V) = \{\id_V\}$. (See eg. \cite{Br18}, \cite{Br19} for more details.)
\item We denote by $\kappa (V)$ the Kodaira dimension of a smooth complex projective variety $V$. Then $\kappa (V) \in \{-\infty, 0, 1, \ldots, n-1, n\}$, where $n = \dim V$. The Kodaira dimension is a birational invariant in the sense that $\kappa (V) = \kappa (V')$ if $V$ and $V'$ are smooth birational projective varieties. (See eg. \cite{Ue75} for more details.)
\end{enumerate}
\end{definition}

The following is our main theorem: 

\begin{theorem}\label{thm1}
\begin{enumerate}
\item For each integer $n \ge 3$, there is a smooth complex projective rational variety $V$ of dimension $n$, with discrete, not finitely generated $\Aut (V)$, and with infinitely many mutually non-isomorphic real forms. 

\item Let $V$ be a smooth complex projective variety of dimension $n \ge 3$. If $\Aut (V)/\Aut_0(V)$ is not finitely generated, then $\kappa (V) \in \{-\infty, 0, 1, \ldots, n-2 \}$. 

\item Conversely, for each pair of integers $n \ge 3$ and $\kappa \in \{-\infty, 0, 1, \ldots, n-2 \}$, there is a smooth complex projective variety $V$ of dimension $n$ and of Kodaira dimension $\kappa$, with discrete, not finitely generated $\Aut (V)$, and with infinitely many mutually non-isomorphic real forms. 
\end{enumerate}
\end{theorem}

Our proof of Theorem \ref{thm1} (1) and (3) is explicit and is based on the surfaces constructed in \cite{Le18} and \cite{DO19}. As in \cite{Le18} and \cite{DO19}, the most crucial part of the construction is a realization of some non-finitely generated discrete subgroup $G$ of $\Aut (S)$ of some special surface $S$ as a finite index subgroup of the automorphism group $\Aut (V)$ of another variety $V$ via taking some products and suitable blowing-ups, so that $V$ keeps the group $G$ as automorphisms but kills almost all $\Aut (S) \setminus G$ and at the same time produces essentially no new automorphisms. This process is, in general, hardest for rational varieties compared with other varieties, especially because of the last requirement "{\it $V$ produces essentially no new automorphisms}" (cf. \cite[Page 198, Rem.4]{Le18}). 

We are primarily interested in smooth complex projective varieties. However, concerning the base field of the non-finite generation part of Theorem \ref{thm1}, it might be worth mentioning the following:
 
\begin{remark}\label{rem1} Let $p$ be a prime number and $k$ be an algebraically closed field containing the rational function field $\F_p(t)$. In the proof of Theorem \ref{thm1} (1), we will use a special rational surface $S$ defined over $\R$ constructed by Lesieutre \cite{Le18}. (See Section \ref{sect3}.) Replacing $S$ by a rational surface defined over $\F_p(t)$ in \cite[Page 203]{Le18}, we find that for each $n \ge 3$ and for each prime number $p \ge 3$, there is a smooth projective rational variety $V$ of dimension $n$ defined over $k$, with discrete, not finitely generated $\Aut (V)$. Indeed, the construction and proof of Section \ref{sect3} is valid if we replace both $\R$ and $\C$ by $k$. 
\end{remark} 

By Theorem \ref{thm1} and \cite{DO19}, the most major remaining open problem for Question \ref{quest1} is now the following:

\begin{question}\label{quest2}
\begin{enumerate}
\item Is there a smooth complex projective rational surface $V$ with discrete, not finitely generated $\Aut (V)$? 

\item Is there a smooth complex projective rational surface $V$ with infinitely many mutually non-isomorphic real forms? 

\end{enumerate}
\end{question}

Unfortunately, our method is not available to answer Question \ref{quest2}. See \cite{Ben16}, \cite{Ben17} for some constraint from complex dynamics.

As in \cite{Le18} and \cite{DO19}, throughout this paper, we use the following three general facts frequently. See eg. \cite[Page 181, Cor.1]{Su82} for Theorem \ref{thm2}, \cite[Th.14.10]{Ue75} for Theorem \ref{thm3} and see \cite[Lem.13]{Le18} for Theorem \ref{thm4}. 

\begin{theorem}\label{thm2} Let $G$ be a group and let $H$ be a subgroup of $G$ such that $[G:H] <\infty$. Then $G$ is finitely generated if and only if $H$ is finitely generated. 
\end{theorem}

\begin{theorem}\label{thm3} Let $V$ be a smooth complex projective variety and let ${\rm Bir}\, (V)$ be the group of the birational automorphisms of $V$. Then $|{\rm Im}\, (\rho_m)| < \infty$ for all $m \ge 0$, where $\rho_m$ is the (contravariant) group homomorphism
$$\rho_m : {\rm Bir}\, (V) \to {\rm GL}\, (H^0(V, mK_V))\,\, ;\,\, f \mapsto f^*.$$  
\end{theorem}

Throughout this paper, we denote by $c$ the complex conjugate map. Then $c$ is the generator of the Galois group ${\rm Gal}\,(\C/\R)$ and ${\rm Gal}\,(\C/\R) = \{\id, c\}$.
 
\begin{theorem}\label{thm4} Let $V$ be a smooth projective complex variety defined over $\R$. Suppose that there is a finite index subgroup $G$ of $\Aut (V)$ 
such that ${\rm Gal}\,(\C/\R) = \{\id, c\}$ acts on $G$ as identity via $g \mapsto c \circ g \circ c$ and $G$ has infinitely many conjugacy classes of involutions. Then $V$ has infinitely many  mutually non-isomorphic real forms.
\end{theorem}

For a complex projective variety $X$ and non-empty closed algebraic subsets $Y_i$ ($i \in I$) of $X$, we define
$${\rm Aut}\, (X, Y_i\,\, (i \in I)) := \big\{f \in {\rm Aut}\, (X)\, |\, f(Y_i) = Y_i\,\, (i \in I) \big\}.$$
This is a subgroup of ${\rm Aut}\, (X)$ and $f_{|Y_i} \in {\rm Aut}\, (Y_i)$ ($i \in I$) if $f \in {\rm Aut}\, (X, Y_i\,\, (i \in I))$. 
For simplicity, we denote the group ${\rm Aut}\, (X, \{P\})$ by ${\rm Aut}\, (X, P)$ if $P$ is a closed point of $X$. 

Whenever we consider a complex variety V with a natural real form (which will be understood by the construction in our case), we denote it by $V_{\R}$. By abuse of notation, we denote by the set of real points $V_{\R}(\R)$ of $V_{\R}$ simply by $V(\R)$ and regard it as a subset of the set of closed points of $V$. More precisely, if $v$ is a closed point of $V$, i.e., if $v \in V(\C)$, then we denote $v \in V(\R)$ exactly when $c(v) = v$ under the complex conjugate map $c$ of $V$ with respect to the real form $V_{\R}$. 

\medskip\noindent
{\bf Acknowledgements.} We would like to express our thanks to Professor Shigeru Mukai for inspiring discussions, especially, asking us a construction of Lesieutre, which brought us new idea for our construction in this paper. We would like to express our thanks to Professor Michel Brion for his very careful reading and many valuable comments. This work has been started during the first and third named authors' stay at the University of Tokyo in December 2019. We would like to express our thanks to the University of Tokyo for hospitalities.

\section{Lesieutre's surface}\label{sect2}
In this section, we recall from \cite{Le18} the core rational surface, which we call Lesieutre's surface. Lesieutre's surface will play a crucial role in our proof of Theorem \ref{thm1} (1). 

Let $L_i'$ ($0 \le i \le 5$) be six lines defined over $\R$ in $\P^2$ such that the intersection points $P_{ij} := L_i' \cap L_j'$ ($0 \le i \not= j \le 5$) are mutually distinct and the points $P_{ij}$, $P_{kl}$, $P_{mn}$ are not colinear for any partition 
$$\{0, 1, 2, 3, 4, 5\} = \{i, j\} \cup \{k, l\} \cup \{m, n\}.$$
We choose such six lines so that
$$P_{10} = 0\,\, ,\,\, P_{20} = 1\,\, ,\,\,P_{30} = 2\,\, ,\,\, P_{40} = 3\,\, ,\,\, P_{50} = \infty$$
under a fixed affine coordinate $x$ of $L_0' = \P^1$. 
Let $S \to \P^2$ be the blow-up of $\P^2$ at the 15 points $P_{ij}$. 

We denote by $E_{ij} \subset S$ the exceptional curve over $P_{ij}$ and by $L_i \subset S$ the proper transform of $L_i'$. Set 
$$C := L_0\,\, ,\,\, P_{i} := C \cap E_{i0}\,\, (1 \le i \le 5).$$ 
Note that $L_i = \P^1$ and $(L_i, L_i) = -4$. Under the identification $C = L_0'$ via $S \to \P^2$, we may use the same affine coordinate $x$ for $C = \P^1$ as $L_0'$. Then
$$P_{1} = 0\,\, ,\,\, P_{2} = 1\,\, ,\,\,P_{3} = 2\,\, ,\,\, P_{4} = 3\,\, ,\,\, P_{5} = \infty$$
with respect to the coordinate $x$. 

\begin{definition}\label{def21}
We call this surface $S$ Lesieutre's surface. By construction, $S$ is defined over $\R$, i.e., 
$$S = S_{\R} \times_{\Spec \R} \Spec \C\,\, ,$$ 
where $S_{\R}$ is the blow  up of 
$\P_{\R}^2 := {\rm Proj}\, \R[x_0, x_1, x_2]$ 
at the $\R$-rational points $P_{ij} \in \P_{\R}^2(\R)$.  In order to distinguish with other real forms, we call this $S_{\R}$ the natural real form of $S$.  
\end{definition}
 By definition, Lesieutre's surface is a smooth projective rational surface defined over $\R$. 
\begin{proposition}\label{prop21} Let $S$ be Lesieutre's surface. Then:
\begin{enumerate}
\item $|-2K_S| = \{\sum_{i=0}^{5} L_i\}$.
\item $\Aut (S)$ is discrete. More strongly, the contravariant group homomorphism
$$\Aut (S) \to \Aut ({\rm Pic}(S)) = \Aut ({\rm NS}(S))\,\, ;\,\, f \mapsto f^*$$ is injective.
\item $\Aut (S, P_5) = \Aut(S, C, P_5)$. 
\item Every element of $\Aut (S)$ is defined over $\R$ with respect to the natural real form $S_{\R}$. In particular, the Galois group ${\rm Gal}(\C/\R)$ acts on $\Aut (S)$ as identity.
\end{enumerate}
\end{proposition}

\begin{proof} The assertion (1) follows from the adjunction formula and $(L_i, L_i) = -4 <0$. The assertion (2) is proved by \cite[Thm.3 (1)]{Le18}. Note that $\Aut(S)$ preserves the divisor $\sum_{i=0}^{5} L_i$ by (1). Then the assertion (3) is clear, because $C = L_0$ is the unique irreducible component of $\sum_{i=0}^{5} L_i$ containing $P_5$. The first part of the assertion (4) is already explained. The second assertion of (4) is proved in the course of proof of 
\cite[Lem.19]{Le18}. We shall reproduce the proof here for the convenience of the readers. Since the curves $E_{ij}$ and $L_i$ are defined over $\R$ and their classes generate ${\rm Pic} (S) = {\rm NS} (S)$, it follows that ${\rm Gal}(\C/\R)$ acts on ${\rm Pic}\, (S)$ as identity. Thus ${\rm Gal}(\C/\R)$ acts on $\Aut (S)$ as identity by (2). Note that the representation in (2) is equivariant under the Galois action. 
\end{proof}

By Proposition \ref{prop21} (3), we have a representation
$$r_C : \Aut(S, P_5) = \Aut (S, C, P_5) \to \Aut(C, P_5) = \{f(x) = a + bx\,|\, a \in \C, b \in \C^{\times}\}\,\, ;\,\, g \mapsto g_{|C}\,\, .$$
Note that $\{f(x) = a \pm x\,|\, a \in \C\})$ 
is a subgroup of $\Aut (C, P_5)$. 

\begin{definition}\label{def22}
$$G := r_C^{-1}(\{f(x) = a \pm x\,|\, a \in \C\}) = \{\varphi \in \Aut (S, C, P_5)\, |\, d(\varphi_{|C})_{P_5} = \pm 1\}\,\, .$$
Here $d(\varphi_{|C})_{P_5}$ is the differential map of $\varphi_{|C} : C \to C$ at $P_5$. 
\end{definition}
The group $G$ is the same group as $G^{\pm}$ in \cite[Page 204]{Le18}. Note that $r_C(G)$ is much smaller than the group $\{f(x) = a \pm x\,|\, a \in \C\}$. 

\begin{proposition}\label{prop22} The group $G$ satisfies:
\begin{enumerate}
\item ${\rm Im}(r_C)$ (resp. $r_C(G)$) contains the following elements $f_1$, $f_2$, $f_3$ (resp. $f_1, f_3$):
$$f_1(x) = x+1\,\, ,\,\, f_2(x) = 2x\,\, ,\,\, f_3(x) = -x\,\,.$$  
\item $G$ is not finitely generated.
\item $G$ has infinitely many conjugacy classes of involutions. 
\end{enumerate}
\end{proposition}
\begin{proof} The fact that $f_1, f_2 \in {\rm Im}(r_C)$ follows from \cite[Lem.4]{Le18}. Set $f_4(x) = 3-x$. Then $f_4 \in r_C(G)$ by \cite[2nd paragraph, Page 204]{Le18}. Since $f_4 \circ f_1^{3} (x) = -x$, it follows that $f_3 \in r_C(G)$. This proves the assertion (1). 

We show the assertion (2). The group
$$G^{+} := r_C^{-1}(\{f(x) = a + x\,|\, a \in \C\})$$
is a subgroup of index two of $G$. So, by Theorem \ref{thm2}, it suffices to show that $G^{+}$ is not finitely generated. 

Observe that $r_C(G^{+})$ is an abelian group, as it is a subgroup of the abelian group 
$$\{f(z) = z+a\,|\, a \in \C\} \simeq (\C, +).$$ 
and 
$$f_2^{-n} \circ f_1 \circ f_2^{n} (x) = x + \frac{1}{2^n} \cdot$$
Hence $r_C(G^{+})$ has a subgroup
$$\langle f_2^{-n} \circ f_1 \circ f_2^{n}\, |\, n \in \Z \rangle \simeq \langle \frac{1}{2^n}\, |, n \in \Z \rangle .$$
This subgroup is not finitely generated. Thus the abelian group $r_C(G^{+})$ is not finitely generated, either. Hence $G^{+}$ is not finitely generated. 

We show the assertion (3). As in \cite[Page 204]{Le18}, we consider the subgroup $G_{{\rm ev}}$ of $G$ defined by
$$\{f \in G\,|\, f(L_0) = L_0, f(L_5) = L_5, f(L_1\cup L_4) = L_1 \cup L_4, f(L_2 \cup L_3) = L_2 \cup L_3\}\,\, .$$
Since $f(\sum_{i=0}^{5} L_i) = \sum_{i=0}^{5} L_i$ for any $f \in \Aut (S)$ by Proposition \ref{prop21} (1), $G_{{\rm ev}}$ is a finite index normal subgroup of $G$. On the other hand, it is shown by \cite[Cor.18]{Le18} that $G_{{\rm ev}}$ contains infinitely many conjugacy classes of involutions. Then $G_{{\rm ev}}$ has infinitely many classes of involutions under the conjugate action of $G$ on $G_{{\rm ev}}$, as $G_{{\rm ev}}$ is a finite index normal subgroup of $G$. Hence $G$ has infinitely many conjugacy classes of involutions as well.
\end{proof}

\begin{definition}\label{def23}
Let $S$ be Lesieutre's surface. We choose and fix $\tau_S \in G$ such that $r_C(\tau_S) = f_3$, that is, $r_C(\tau_S)(x) = -x$ on $C = \P^1$. 
\end{definition}

\section{Proof of Theorem \ref{thm1} (1)}\label{sect3} 

We shall prove Theorem \ref{thm1} (1).

Construction \ref{const31} and Proposition \ref{prop31} below will complete the proof of Theorem \ref{thm1} (1). 

We employ the same notations for Lesieutre's surface as in Section \ref{sect2}.

In the rest, the following elementary lemmas will be used frequently.

\begin{lemma}\label{lem31} Let $Y$ and $Z$ be complex projective varieties and let $G$ be a subgroup of $\Aut(Y \times Z)$. Assume that $\Aut (Y)$ is discrete and the projection $Y \times Z \to Z$ is equivariant with respect to $G$. Then $G \subset \Aut(Y) \times \Aut(Z)$. 
\end{lemma}
\begin{proof} Let $f \in G$. By the second assumption, $f$ is of the form 
$$f(y, z) = (f_z(y), f_Z(z)),$$
where $f_Z \in \Aut (Z)$ and $f_z \in \Aut (Y)$. Then we have the morphism
$$Z \to \Aut(Y)\,\, ;\,\, z \mapsto f_z.$$
Since $\Aut (Y)$ is discrete by the first assumption, it follows that $f_z$ does not depend on $z \in Z$. Hence $f = (f_Y, f_Z)$ for some $f_Y \in \Aut(Y)$. This implies the result. 
\end{proof}

\begin{lemma}\label{lem32} Let $f:\P^m\to X$ be a morphism where $X$ is a projective variety of dimension $< m$. Then $f$ is constant.
\end{lemma}

\begin{proof} Let $H$ be a very ample divisor on $X$. Since $(f^*H)^m =0$ and $f^*H$ is effective, it follows that $f^*H = 0$ in ${\rm Pic}\, (\P^m) \simeq \Z$. This implies the result. 
\end{proof}

\begin{lemma}\label{lem33} Let $A$ be a finite subset of $\P^m$ containing $m+2$ points in general position in the sense that no $m+1$ points among these $m+2$ points are contained in a hyperplane of $\P^m$. Then $\Aut(\P^m,A)$ is finite.
\end{lemma}

\begin{proof} 
The group $G$ of all automorphisms which fix each point of $A$ is a finite-index subgroup of  $\Aut(\P^m,A)$. It is enough to show that $G$ is trivial. 
This is true because if $f$ is an automorphism then it is given by a square matrix of size $m+1$. It has at most $m+1$ linearly independent eigenvectors.
\end{proof}

The following generalization has its own interest. We will also apply it for abelian varieies in Section \ref{sect5}.

\begin{lemma}\label{lem34} 
Let $X$ be any compact K\"ahler manifold of dimension $n$. There is a number $N$ such that if $A$ is a finite subset of $X$ containing $N$ points in general position, then $\Aut(X,A)$ is finite. In particular, when all morphisms $\P^{n-1}\to X$ are constant (e.g. $X$ is a complex torus), then $\Aut(\widehat X)$ is finite, where $\widehat X$ is the blow-up of $X$ at the points in $A$.
\end{lemma}
\begin{proof}
The second assertion is a consequence of the first one because for such an $X$ we have 
$$\Aut(\widehat X) = \Aut(\widehat X, E_A) = \Aut(X,A),$$
where $E_A$ is the set of exceptional divisors of $\widehat X \to X$.  

The first assertion is a consequence of Fujiki-Lieberman's theorem (\cite[Thm.4.8]{Fu78}, \cite{Li78}). Indeed, since $\Aut(X)$ is a complex Lie group of finite dimension and $\Aut_0(X)$ is associated to holomorphic vector fields of $X$, if $P_1\in X$ is a general point, then $\Aut(X,P_1)$ has dimension smaller than the one of $\Aut(X)$. By induction, there exists $N$ such that for general $P_1,\ldots,P_{N-1}$, the group $\Aut(X, P_1,\ldots,P_{N-1})$ is discrete. It follows that the set of points which are fixed by some non-trivial element of this group is a countable union of proper analytic subsets of $X$. Choose $P_N \in X$ outside this set. Then we have that 
$\Aut(X,P_1,\ldots,P_{N-1}, P_{N})$ is finite. Hence so is $\Aut(X,\{P_1,\ldots,P_{N-1}, P_{N}\})$, because 
$$[\Aut(X,\{P_1,\ldots,P_{N-1}, P_{N}\}) : \Aut(X,P_1,\ldots,P_{N-1}, P_{N})] \le N!.$$ 
\end{proof}

\begin{construction}\label{const31}

Let $n \ge 3$ and set $m = n-2$. Let $\P_{\R}^{m} = {\rm Proj}\, \R[x_0, \ldots, x_m]$ be the projective space defined over $\R$ and regard $\P_{\R}^{m}$ as the natural real form of the complex projective space $\P^m$. We fix the affine coordinates $z_i = x_i/x_0$ and denote by $R_0 :=(0)_{i=1}^{m} \in \P^{m}(\R)$ the origin with respect to the affine coordinates $(z_i)_{i=1}^{m}$. (See the end of Introduction for the precise meaning of $\P^m(\R)$.) Let $\iota$ be the involution of $\P^{m}$ defined by $\iota((z_i)_{i=1}^{m}) = (-z_i)_{i=1}^{m}$. Note that $\iota$ is defined over $\R$. 

Let us choose a finite set 
$$R := \{R_0, R_1, \ldots, R_{2(m+1)}\} = \{R_0, R_1, \ldots, R_{m+1}\,\, ,\,\, \iota(R_1)\,\, ,\,\, \ldots \,\, ,\,\, \iota(R_{m+1})\} \subset \P^{m}(\R)$$
such that $R_i$ ($0 \le i \le m+1$) are in general position in the sense that no $m+1$ points of them are contained in a hyperplane of $\P^m$. 
Then $R$ is invariant under $\iota$ and $\iota \in \Aut(\P^m, R_0, \{R_i\}_{i=1}^{2(m+1)})$. 
Let 
$$X_0 := S \times \P^{m},$$
where $S$ is Lesieutre's surface. Then $X_0$ is a smooth projective variety of dimension $n = m+2$ defined over $\R$ with the natural real form $X_{0, \R} = S_{\R} \times \P_{\R}^m$. 

We will use the same notations of the points and curves on $S$ as in Section \ref{sect2}. 

Let 
$$\pi_1 : X_1 \to X_0$$
be the blow-up of $X_0$ at the points in $\{P_5\} \times R \subset X_{0}(\R)$. (Once again, see the end of Introduction for the precise meaning of $X_0(\R)$.) We denote by $T_{(P_5,R_0)}X_0$ the tangent space of $X_0$ at $(P_5,R_0)$. Denote also by
$$E_0 = \P(T_{(P_5, R_0)}X_0) = \P^{m+1} \subset X_1$$ 
the exceptional divisor corresponding to the point $(P_5, R_0) \in X_0$ and by $E_i$ ($1 \le i \le 2(m+1)$) the remaining $2(m+1)$ exceptional divisors. Then $X_1$ and $E_0$ are defined over $\R$ with natural real forms $X_{1, \R}$ and $E_{0, \R}$ .  We choose 
$$0 \not= v \in T_{P_5}C \subset T_{P_5}S\,\, ,\,\, 0 \not= w \in  T_{R_0}\P^{m}$$ such that the point $(v, w) \in T_{(P_5, R_0)}X_0$ defines the point $$[(v, w)] \in E_0(\R) \subset X_1(\R).$$ Let 
$$\pi_2 : X_2\to X_1$$
be the blow-up at the point $[(v,w)]$ in $X_1(\R)$. Then $X_2$ is defined over $\R$ with a natural real form $X_{2, \R}$ induced by $X_{1, \R}$. We denote the exceptional divisor of $\pi_2$ by $F$.   
\end{construction}

\begin{proposition}\label{prop31} 
Let $X_2$ be as in Construction \ref{const31}. Then:
\begin{enumerate}
\item $X_{2}$ is a smooth complex projective rational variety of dimension $n = m+2 \ge 3$ defined over $\R$. 
\item $\Aut (X_{2})$ is discrete and not finitely generated.
\item $X_{2}$ has infinitely many mutually non-isomorphic real forms.
\end{enumerate}
\end{proposition}

\begin{proof} Set $X = X_{2}$. We shall employ the same notation as in Construction \ref{const31}. 

The assertion (1) is clear by the construction. 

We show the assertions (2) and (3) by dividing the argument into several steps.

\begin{claim}\label{claim30}
$\Aut(X_0) = \Aut(S) \times \Aut(\P^m)$. 
\end{claim}

\begin{proof} Recall that $X_0 = S \times \P^{m}$ and $H^0(X_0, -2K_{X_0}) = H^0(S, -2K_S) \otimes H^0(\P^m, -2K_{\P^m})$ by the K\"unneth formula. Since the linear system $|-2K_S|$ consists of a single element by Proposition \ref{prop21} (1), while $-2K_{\P^{m}}$ is very ample, the anti-bicanonical map
$$\Phi_{|-2K_{X_0}|} : X_0 \to \P^m$$
coincides with the second projection $p_2 : X_0 \to \P^m$. 
Since the linear system $|-2K_{X_0}|$ is preserved by $\Aut (X_0)$, it follows that the second projection $p_2 : X_0 \to \P^m$ is $\Aut (X_0)$-equivariant. Since $\Aut(S)$ is discrete by Proposition \ref{prop21}, the result follows from Lemma \ref{lem31}. 
\end{proof}

\begin{claim}\label{claim31}
\begin{enumerate}
\item There is no non-constant morphism $\varphi : \P^{m+1} \to X_0$. 
\item Let $\varphi: \P^{m+1} \to X_1$ be a non-constant morphism. Then 
$\varphi(\P^{m+1})$ is one of the irreducible components of $\pi_1^{-1}(\{P_5\} \times R)$, i.e., one of the exceptional divisors $E_i$ ($0 \le i \le 2(m+1)$). 
\item Let $\varphi: \P^{m+1} \to X_2$ be a non-constant morphism. Then 
$\varphi(\P^{m+1})$ is one of the following divisors:
$$F\,\, ,\,\, \pi_2^{-1}(E_i) \simeq E_i\,\, (1 \le i \le 2(m+1)).$$  
\end{enumerate}
\end{claim}

\begin{proof} We show the assertion (1). Note that $m+1 \ge 2$. Since the Picard number $\rho (S) \ge 2$, there is no surjective morphism $\P^{m+1} \to S$ if $m+1 = 2$. Therefore there is no non-constant morphism $\P^{m+1} \to S$ or $\P^{m+1} \to \P^m$ by Lemma \ref{lem32}. Hence the morphism $p_i \circ \varphi$ is constant for the projections $p_i$ ($i = 1,  2$) from $X_0 = S \times \P^m$ to the $i$-th factor. Hence $\varphi$ is constant. 

Since $\pi_1 \circ \varphi$ is constant by (1), the assertion (2) follows. 

We show the assertion (3). Recall that the proper transform $E_0'$ of $E_0$ on $X_2$ is the blow-up of $E_0 = \P^{m+1}$ at the point $[(v,w)]$. Hence $E_0'$ is of Picard number $\rho(E_0') = 2$. Hence there is no surjective morphism 
$\P^{m+1} \to E_0'$. Therefore, by Lemma \ref{lem32}, $E_0'$ admits no non-constant morphism from $\P^{m+1}$. This together with the assertion (2) implies the assertion (3) exactly for the same reason as in the proof of (2).
\end{proof}

From now, we regard the subgroups of $\Aut (X)$ and $\Aut(X_1)$ as subgroups of ${\rm Bir}\, (X_0)$ via the birational morphisms $\pi_1$ and $\pi_2$. For instance, we say that $G_1 = G_2$ (resp. $G_1 \subset G_2$) for a subgroup $G_1 \subset \Aut (X)$ and a subgroup $G_2 \subset \Aut(X_1)$ if $G_1 =G_2$ (resp. $G_1 \subset G_2$) in ${\rm Bir}\, (X_0)$. We also identify $\Aut (X_0) = \Aut(S) \times \Aut(\P^m)$ by Claim \ref{claim30}. 

\begin{claim}\label{claim32} 
\begin{enumerate}
\item 
$\Aut(X_1) = \Aut(X_0, \{P_5\} \times R) = \Aut(S, C, P_5) \times \Aut(\P^m, R)$.
\item 
$\Aut(X_1, [(v,w)]) = \Aut(X_1, [(v,w)], E_0) \subset \Aut(S, C, P_5) \times \Aut(\P^m, R_0, \{R_i\}_{i=1}^{2(m+1)})$.
\end{enumerate}
\end{claim}

\begin{proof}
The assertion (1) follow from Claim \ref{claim31} (2). 
Since $[(v,w)]\in E_0$ and $[(v,w)]\notin E_i$ for $i \ge 1$, the assertion (2) follows from the assertion (1).
\end{proof}

From now, we use $\Aut(X_1, [(v,w)]) \subset \Aut(S, C, P_5) \times \Aut(\P^m, R_0, \{R_i\}_{i=1}^{2(m+1)})$ and denote an element of $\Aut(X_1, [(v,w)])$ as in the form:
$$(\varphi, g) \in \Aut(S, C, P_5) \times \Aut(\P^m, R_0, \{R_i\}_{i=1}^{2(m+1)}).$$ 
Let $\epsilon \in \{0, 1\}$. We define
$$H := \{(\varphi, \iota^{\epsilon}) \in \Aut(S, C, P_5) \times \Aut(\P^m, R_0, \{R_i\}_{i=1}^{2(m+1)})\,|\, d(\varphi_{|C})_{P_5}(v) = (-1)^{\epsilon} v\}.$$
Here $\iota$ is the involution defined in Construction \ref{const31} and $d(\varphi_{|C})_{P_5}$ is the differential map of $\varphi_{|C} : C \to C$ at $P_5$. By definition, the index $\epsilon$ in $(\varphi, \iota^{\epsilon}) \in H$ is uniquely determined by $\varphi$. 

\begin{claim}\label{claim33} $H$ is a finite index subgroup of $\Aut(X_1, [(v,w)])$ and $H =G$ under the identification $(\varphi, \iota^{\epsilon}) = \varphi$. Here $G$ is the group in Definition \ref{def22}.  
\end{claim}

\begin{proof} Let 
$$H' := \{(\varphi, g) \in \Aut(S, C, P_5) \times \Aut(\P^m, R_0, \{R_i\}_{i=1}^{2(m+1)})\, |\, (\varphi,g)(\C(v, w)) = \C(v, w)\}.$$
Here we recall that $\C(v, w)$ is the $1$-dimensional linear space in $T_{(P_5,R_0)}X_0$ spanned by $(v, w)$ and the action of $(\varphi, g)$ on $\C(v, w)$ is nothing but the differential map. Then, by Claim \ref{claim32}, we have
$$\Aut(X_1, [(v,w)]) = H'.$$
Since $\iota^{\epsilon} (w) = (-1)^{\epsilon} w$ and $w \not= 0$, the condition $(\varphi, \iota^{\epsilon})\C(v, w) = \C(v, w)$ is equivalent to the condition that $d(\varphi_{|C})_{P_5}(v) = (-1)^{\epsilon} v$. Then $H$ is a subgroup of $H'$. Since the $m+2$ points $\{R_0\} \cup \{R_i\}_{i=1}^{m+1} \subset \P^m$ are in general position, by Lemma \ref{lem33}, $\Aut(\P^m, \{R_0\}\cup\{R_i\}_{i=1}^{m+1})$ is a finite group. Hence so is  $\Aut(\P^m, R_0, \{R_i\}_{i=1}^{2(m+1)})$, because $$\Aut(\P^m, R_0, R_1,...,R_{2(m+1)})\subset \Aut(\P^m, R_0, \{R_i\}_{i=1}^{m+1})\subset \Aut(\P^m, \{R_0\}\cup\{R_i\}_{i=1}^{m+1})$$ and $$[\Aut(\P^m, R_0, \{R_i\}_{i=1}^{2(m+1)}):\Aut(\P^m, R_0, R_1,...,R_{2(m+1)})]<\infty.$$ In particular, the number of $g$'s in the definition of $H'$ is at most finite. Thus $[H' : H] < \infty$. 

The last assertion is clear by the definitions of $G$ and $H$ with the remark before Claim \ref{claim33}. This proves the claim. 
\end{proof}

\begin{claim}\label{claim34}
\begin{enumerate}
\item  
$\Aut(X_1, [(v,w)])$ is a finite index subgroup of $\Aut (X)$.
\item 
$H$ is a finite index subgroup of $\Aut (X)$.
\end{enumerate}
\end{claim}

\begin{proof} By Claim \ref{claim31} (3), we have
$$\Aut (X) = \Aut (X, \{\pi_2^{-1}(E_i), F\,|\, 1 \le i \le 2(m+1)\}).$$
On the other hand, by construction, $\Aut(X_1, [(v,w)]) = \Aut(X, F) \subset \Aut(X)$.
Since $\{\pi_2^{-1}(E_i), F\,|\, 1 \le i \le 2(m+1)\}$ is a finite family, this implies the assertion (1). The assertion (2) follows from (1) and Claim \ref{claim33}.  
\end{proof}

Now we are ready to complete the proof of Proposition \ref{prop31} (2), (3). 

By Claims \ref{claim33} and \ref{claim34} (2), $H \simeq G$ is a finite index subgroup of $\Aut (X)$. Since $G$ is not finitely generated by Proposition \ref{prop22} (2), $\Aut(X)$ is not finitely generated as well by Theorem \ref{thm2}. This proves Proposition \ref{prop31} (2).

By the construction, $X$ is defined over $\R$. By Proposition \ref{prop21} (4) and by the construction, the Galois group ${\rm Gal}(\C/\R)$ acts trivially on $H$. Since $G$ has infinitely many conjugacy classes of involutions by Proposition \ref{prop22} (3), the same holds for $H$ because $H \simeq G$. Since $H$ is a finite index subgroup of $\Aut (X)$, it follows from Theorem \ref{thm4} (\cite[Lem.13]{Le18}) that $X$ has infinitely many mutually non-isomorphic real forms. This proves Proposition \ref{prop31} (3).
\end{proof}

\section{Proof of Theorem \ref{thm1} (2)}\label{sect4}

In this section, we prove Theorem \ref{thm1} (2). Let $V$ be a smooth complex projective variety of dimension $n$. 

Consider the case where $\kappa (V) = n$. Then the pluricanonical map $\Phi_{|mK_V|}$ for large divisible $m$ is a birational map onto the image. Thus $\Aut (V)$ is a finite group by Theorem \ref{thm3}.

Next, consider the case where $\kappa (V) = n-1 \ge 1$. Then the geometric generic fibre $V_{\overline{\eta}}$ of the pluricanonical map 
$$\Phi_{|mK_V|} : V \dasharrow B$$ 
for large divisible $m$ is an elliptic curve defined over $\overline{\C(B)}$, an algebraic closure of the function field $\C(B)$. (See eg. \cite[Chap.IV, Sect.4]{Ha77} for basic properties of elliptic curves over an algebraically closed field, which we will use from now on.) By Theorem \ref{thm3}, there is a subgroup $G$ of $\Aut(V)$ such that 
$$[\Aut(V) : G] < \infty\,\, ,\,\, \Aut_0 (V) \subset G \subset \Aut(V_{\overline{\eta}}/\overline{\C(B)}).$$   
Set 
$$A := \Aut_0(V_{\overline{\eta}}/\overline{\C(B)}).$$ 
Since $V_{\overline{\eta}}$ is an elliptic curve over $\overline{\C(B)}$, 
the group $A$ is an abelian group consisting of translations and $\Aut(V_{\overline{\eta}}/\overline{\C(B)})$ is a semi-direct product of $A$ and some finite cyclic group $\Z/a\Z$. Thus $A$ is an abelian subgroup of $\Aut(V_{\overline{\eta}}/\overline{\C(B)})$ such that $[\Aut(V_{\overline{\eta}}/\overline{\C(B)}) : A] < \infty$.  
Hence 
$$A' := A \cap G$$
is also an abelian group and $[G:A'] <\infty$. Then $[\Aut (V) : A'] < \infty$ 
as well. Consider the contravariant group homomorphism 
$$\rho : \Aut(V) \to {\rm GL}({\rm NS}\,(V)/({\rm torsion}))\,\, ;\,\, f \mapsto f^*.$$
Since ${\rm GL}({\rm NS}\,(V)/({\rm torsion})) \simeq {\rm GL}\, (N, \Z)$ for some $N$ and $A'$ is an abelian group, the group $\rho(A')$ is isomorphic to a solvable subgroup of ${\rm GL}\, (N, \Z)$. In particular, $\rho(A')$ is  finitely generated by the famous theorem of Malcev (see eg. \cite[Chap.2]{Se83}).

Since $[\rho(\Aut(V)) : \rho(A')] < \infty$ by $[\Aut (V) : A'] < \infty$, the group $\rho(\Aut(V))$ is finitely generated as well by Theorem \ref{thm2}. Since $\rho(\Aut(V)) \simeq \Aut (V)/{\rm Ker}\, \rho$, the group $\Aut (V)/{\rm Ker}\, (\rho)$ is also finitely generated.  
Since $[{\rm Ker}\, \rho : \Aut_0(V)] < \infty$ by an algebraic version of Fujiki-Lieberman's theorem (see eg. \cite[Thm.2.10]{Br19}), $\Aut(V)/\Aut_0(V)$ is also finitely generated by the following exact sequence of groups:
$$1 \to {\rm Ker}\, \rho/\Aut_0(V) \to \Aut(V)/\Aut_0(V) \to \Aut (V)/{\rm Ker}\, \rho \to 1.$$

This proves Theorem \ref{thm1} (2).      

\section{Proof of Theorem \ref{thm1} (3)}\label{sect5}

We prove Theorem \ref{thm1} (3). 

Note that $\kappa (V) = \kappa (\P^n) = -\infty$ if $V$ is a smooth complex projective rational variety. Therefore, the case $\kappa (V) = -\infty$ follows from Theorem \ref{thm1} (1). 

From now, we consider the case where $\kappa (V) \ge 0$. For this, instead of Lesieutre's surface, we use the surface $S_2$ constructed by 
\cite[Sect.4]{DO19} 
to construct the desired varieties. 

In the rest, we denote $M := S_2$. The surface $M$ is constructed from a Kummer K3 surface of product type in \cite[Sect.4]{DO19}. Since, we will not use the explicit form of $M$, we omit to repeat the detailed construction and just surmarize basic properties of the surface $M$ we will use. See \cite[Sect.4]{DO19} for the explicit form of $M$. 

\begin{proposition}\label{prop51} There are a smooth projective complex surface $M$ defined over $\R$ with a natural real form $M_{\R}$ and a subgroup $H$ of $\Aut (M)$ such that
\begin{enumerate}
\item $M$ is birational to a smooth complex projective K3 surface, in particular, $\kappa (M) = 0$ and $\Aut (M)$ is discrete;
\item $[\Aut (M):H] < \infty$ and $H$ is not finitely generated. Moreover, there is a point $P \in M(\R)$ such that $h(P) = P$ for all $h \in H$;
\item $c \circ h \circ c = h$ for every element of $H$ under the Galois action of ${\rm Gal}\, (\C/\R) = \{\id, c\}$ with respect to the real form $M_{\R}$; 
and
\item $H$ has infinitely many conjugacy classes of involutions.
\end{enumerate}
\end{proposition}

\begin{proof} As mentioned, we choose the surface $S_2$ in \cite[Sect.4]{DO19} as $M$. We choose the group $H$ in \cite[Lem.4.6]{DO19} as our $H$. Then $M$ and $H$ satisfy the required properties. (1) is clearly satisfied. The first part of (2) follows from \cite[Thm.2.8, Lem.4.6]{DO19} and Theorem \ref{thm2}. For the last part of (2), we may choose one of the two points $P'$ or $P''$ in \cite[Def.2.7]{DO19} as $P$. (3) follows from \cite[Lem.4.6]{DO19} and (4) follows from \cite[Lem.4.5]{DO19}.  
\end{proof}

Now we are ready to complete the proof of Theorem \ref{thm1} (3). 

Let $T$ be a complex abelian variety of dimension $l$, defined over $\R$ as an abstract variety. Let $A$ be a finite subset of $T$ such that $\Aut(T, A)$ is finite and $c(A) = A$ under the complex conjugate map $c$ of $T$ with respect to $T_{\R}$. Such a subset $A$ exists. Indeed, by Lemma \ref{lem34}, there is a finite subset $A' \subset T$ such that $\Aut(T, A')$ is finite. Then we may take $A = c(A') \cup A'$. 

Let 
$$\pi : X_{l} \to M \times T$$
be the blow-up at the points in $\{P\} \times A$. Let $E_i \simeq \P^{l+1}$ ($1 \le i \le |A|$) be the exceptional divisors of $\pi$.   
\begin{claim}\label{claim51} 
\begin{enumerate}
\item $X_l$ is a smooth complex projective variety defined over $\R$ 
and $\dim X_l = l+2$ and $\kappa (X_l) = 0$.
\item $\Aut (X_l)$ is discrete and not finitely generated. Moreover, $X_l$ has infinite many mutually non-isomorphic real forms.
\end{enumerate}
\end{claim} 

\begin{proof} The assertion (1) is clear from the construction. We show the assertion (2). If $l = 0$, then the result follows from Proposition \ref{prop51}. From now, we assume that $l \ge 1$. Let $f \in \Aut (X_l)$. Since $T$ has no rational curve, it follows that 
$$\pi(f(E_i)) \subset M \times A.$$
 Since $f(E_i) \simeq \P^{l+1}$ with $l+1 \ge 2$ and $M$ is not covered by rational curves by $\kappa (M) = 0$, it follows that $\pi(f(E_i))$ is a point. Thus $f(\{E_i\}_{i=1}^{|A|}) = \{E_i\}_{i=1}^{|A|}$ and therefore 
$$\Aut (X_l) = \Aut(X_l, \{E_i\}_{i=1}^{|A|}) = \Aut(M \times T, \{P\} \times A).$$
Since the Albanese morphism $M \times T \to T$ is preserved by $\Aut (M \times T)$ and $\Aut (M)$ is discrete, it follow from Lemma \ref{lem31} that    
$$\Aut(M \times T) = \Aut(M) \times \Aut(T).$$
Hence
$$\Aut(X_l) = \Aut(M, P) \times \Aut(T, A).$$
Since $\Aut(T, A)$ is finite, $\Aut(M, P) \times \{\id_T\}$ is a finite index subgroup of $\Aut(X_l)$. Hence $H \times \{\id_T\} \simeq H$, where $H$ is the group in Proposition \ref{prop51}, is also a finite index subgroup of $\Aut(X_l)$. Hence $\Aut(X_l)$ is discrete and is not finitely generated by Proposition \ref{prop51} (2) and Theorem \ref{thm2}. Then $X_l$ has infinitely many mutually non-isomorphic real forms by Proposition \ref{prop51} (3), (4) and Theorem \ref{thm4}. 
\end{proof}
Let $Z_{m} \subset \P^{m+1}$ ($m \ge 1$) be a smooth complex hypersurface of degree $m+3$ defined over $\R$. Set 
$$Y_{l+m} := X_l \times Z_{m}.$$
\begin{claim}\label{claim52} 
\begin{enumerate}
\item $Y_{m+l}$ is a smooth complex projective variety defined over $\R$ with $\dim Y_{l+m} = 2+l +m$ and $\kappa (Y_{l+m}) = \kappa (Z_m) = m$.
\item $\Aut (Y_{l+m})$ is discrete and not finitely generated. Moreover, $Y_{l+m}$ has infinite many mutually non-isomorphic real forms.
\end{enumerate}
\end{claim} 
\begin{proof} Again, the assertion (1) is clear from the construction. We show the assertion (2). Since $|K_{X_l}|$ consists of a single element and $K_{Z_{m}}$ is very ample, the canonical map
$$\Phi_{|K_{Y_{l+m}}|} : Y_{l+m} = X_{l} \times Z_m\to Z_{m}$$
coincides with the second projection $p_2 : Y_{l+m} \to Z_{m}$ for the same reason as in the proof of Claim \ref{claim30}. In particular, the second projection $p_2$ is $\Aut(Y_{l+m})$-equivariant. Since $\Aut(X_l)$ is discrete by Claim \ref{claim51}, it follows from Lemma \ref{lem31} that
$$\Aut(Y_{l+m}) = \Aut(X_l) \times \Aut(Z_m).$$
Since $\Aut (Z_m)$ is finite by Theorem \ref{thm3}, as before, the group 
$$H \times \{\id_T\} \times \{\id_{Z_m}\},$$ 
where $H$ is the group in Proposition \ref{prop51}, is a finite index subgroup of $\Aut(Y_{l+m})$ by Claim \ref{claim51}. The result now follows from the same reason as in the last part of the proof of Claim \ref{claim51}. 
\end{proof}
Theorem \ref{thm1} (3) now follows from Claim \ref{claim51} with $l \ge 1$ 
and Claim \ref{claim52} with $l \ge 0$ and $m \ge 1$.

\end{document}